\documentclass[11pt, one side,emlines]{amsart}
\usepackage{amssymb,latexsym,xy,eucal,mathrsfs, graphicx, tikz, amsmath, caption}
\textwidth=15.2cm \textheight=24cm \theoremstyle{plain}
\newtheorem{lemma}{Lemma}[section]
\newtheorem{theorem}[lemma]{Theorem}

\newtheorem{corollary}[lemma]{Corollary}

\theoremstyle{definition}

\numberwithin{equation}{section}  \voffset
-55truept \hoffset -15truept
\begin{document}
	\baselineskip 15truept
	\title{Construction of four completely independent spanning trees on augmented cubes}
	\maketitle 
	\begin{center} 
		
		S. A. Mane, S. A. Kandekar, B. N. Waphare\\  
		{\small Center for Advanced Studies in Mathematics,
			Department of Mathematics,\\ Savitribai Phule Pune University, Pune-411007, India.}\\
		\email{\emph{manesmruti@yahoo.com; smitakandekar54@gmail.com; waphare@yahoo.com}} 
	\end{center} 
\begin{abstract}
	Let $T_1, T_2,.......,T_k$ be spanning trees in a graph $G$. 
If for any pair of vertices $\{u,v\}$ of $G$, 
the paths between $u$ and $v$ in every $T_i$( $1\leq i\leq k$)
do not contain common edges  
and common vertices, 
except the vertices $u$ and $v$, then  $T_1, T_2,.......,T_k$ 
are called completely independent spanning trees in $G$. 
The $n-$dimensional augmented cube, denoted as $AQ_n$,
a variation of the hypercube possesses several embeddable properties that the hypercube and its variations do not possess. 
For $AQ_n$ ($n \geq 6$), construction of $4$ completely independent spanning trees of which two trees with diameters $2n-5$ and two trees with diameters $2n-3$ are given. 
\end{abstract}

\section{Introduction} 
\indent Interconnection networks have been widely studied recently. The architecture of an interconnection networks is usually denoted as an undirected graph $G$. A graph $G$ is a triple consisting of a vertex set $V(G)$, an edge set $E(G)$, and a relation that associates with each edge two vertices called its endpoints\cite{we}. Many useful topologies have been proposed to balance performance and cost parameters. Among them, the hypercube $Q_n$ is one of the most popular topology and has been studied for parallel networks. Augmented cubes are derivatives of hypercubes (proposed by Choudam and Sunitha\cite{cs}) with good geometric features that retain some favorable properties of the hypercubes (since $Q_n \subset AQ_n$), such as vertex symmetry, maximum connectivity, routing and broadcasting procedures with linear time complexity. An $n-$dimensional augmented cube $AQ_n$ can be formed as an extension of $Q_n$ by adding some links. For any positive integer $n$, $AQ_n$ is $(2n-1)$-regular and  $(2n-1)$-connected (except $n = 3$) graph with $2^n$ vertices. Moreover, $AQ_n$ possesses several embeddable properties that the hypercube and its variations do not possess. The main merit of augmented cubes is that their diameters are about half of those of the corresponding hypercubes. \\
A tree $T$ is called a spanning tree of a graph $G$ if $V(T)= V(G)$. Two spanning trees $T_1$ and $T_2$ in $G$ are edge-disjoint if $E(T_1)\cap E(T_2)= \phi$. For a given tree $T$ and a given pair of vertices $u$ and $v$ of $T$, let $P_T(u,v)$ be the set of vertices in the unique path between $u$ and $v$ in $T$. Two spanning trees $T_1$ and $T_2$ are internally vertex disjoint if for any pair of vertices $u$ and $v$ of $V(G)$, $P_{T_1}(u,v)\cap P_{T_2}(u,v)= \{u,v\}$. Finally, the spanning trees $T_1, T_2,.......,T_k$ of $G$ are completely independent spanning trees (CISTs for short) if they are pairwise edge-disjoint and internally vertex disjoint.\\
The study of CISTs was due to the early work of Hasunuma\cite{h1}, he conjectured that there are $k$ CISTs in any $2k-$connected graph. P$\acute{e}$terfalvi\cite{pt} gave counter example to disprove Hasunuma's conjecture. He showed that there exists a $k-$connected graph which does not contain two CISTs for each $k \geq 2$. Pai et al\cite{p1} showed that the results are negative to Hasunuma's conjecture in case of hypercube of dimension $n \in \{10,12,14,20,22,24,26,28,30\}$. Many authors provided a necessary condition of CISTs\cite{a,c1,f, h3,hl}. For more detail work on CISTs and their diameters see \cite{c2, g,h1,h2,h3,h4,h5,m,w,y}.\\
Constructing CISTs has many applications on interconnection networks such as fault-tolerant broadcasting and secure message distribution.In underlying graph of communication networks, we want vertices to be close together to avoid communication delays.\\
Pai and Chang\cite{p2} provided a unified approach for constructing two CISTs in several hypercube-variant networks, in particular for an $n-$dimensional hypercube variant network, the diameters of the constructed CISTs were $2n-1$. They asked about the hypercube variant networks which they studied "how to design algorithms to construct more than two CISTs in high dimensional hypercube-variant networks with smaller diameter?"\\
Motivated by this question, we provide a construction of four CISTs in augmented cube $AQ_n$ ($n \geq 6$) of which two trees with diameters $2n-3$ and two trees with diameters $2n-5$ are constructed. Also, construction of $n-1$ CISTs in augmented cube $AQ_n$ for $n = 3,4,5$ is given which pointed out that the Hasunuma's conjecture does hold in the case of $AQ_n$ for $n = 3,4,5$.  \\
For undefined terminology and notation, see\cite{we}

\section {Preliminaries}

The definition of the $n-$dimensional augmented cube is stated as the following. Let $n \geq 1$ be an integer. The $n$-dimensional augmented cube, denoted by $AQ_n$, is a graph with $2^n$ vertices, and each
vertex $u$ can be distinctly labeled by an $n$-bit binary string, $u = u_1u_2....u_n$. $AQ_1$ is the graph $K_2$ with vertex set $\{0, 1\}$. For
$n \geq 2$, $AQ_n$ can be recursively constructed by two copies of $AQ_{n-1}$, denoted by $AQ^0_{n-1}$
and $AQ^1_{n-1}$, and by adding
$2^n$ edges between $AQ^0_{n-1}$ and $AQ^1_{n-1}$
as follows:\\ Let $V(AQ^0_{n-1}) = \{0u_2....u_n : u_i \in \{0, 1\}, 2 \leq i \leq n\}$ and 
$V(AQ^1_{n-1}) = \{1v_2....v_n : v_i \in \{0, 1\}, 2 \leq i \leq n\}$. A vertex $u = 0u_2....u_n$ of $AQ^0_{n-1}$
is joined to a vertex $v = 1v_2....v_n $ of $AQ^1_{n-1}$
if and only if for every $i$, $2 \leq i \leq n$ either\\
1. $u_i = v_i$; in this case an edge $\langle u, v \rangle$ is called a hypercube edge and we say $v = u^h$, or\\
2. $u_i = \overline{v_i}$; in this case an edge $\langle u, v \rangle$ is called a complement edge and we say $v = u^c$.\\
Let $E^h_n = \{\langle u, u^h \rangle : u \in V(AQ^0_{n-1})\}$ and $E^c_n = \{\langle u, u^c \rangle : u \in V(AQ^0_{n-1})\}$. See Fig.$1$.\\
\begin{center}
\unitlength 1mm 
\linethickness{0.4pt}
\ifx\plotpoint\undefined\newsavebox{\plotpoint}\fi 
\begin{picture}(136.25,55.25)(0,0)
\put(5.5,16.25){\line(0,1){18.25}}
\put(25,16.25){\line(0,1){18.25}}
\put(63.75,16.25){\line(0,1){18.25}}
\put(112,16.25){\line(0,1){18.25}}
\put(47.75,16.25){\line(0,1){18.25}}
\put(86.5,16.25){\line(0,1){18.25}}
\put(134.75,16.25){\line(0,1){18.25}}
\put(25.25,34.25){\line(1,0){22.5}}
\put(64,34.25){\line(1,0){22.5}}
\put(112.25,34.25){\line(1,0){22.5}}
\put(25.25,16.75){\line(1,0){22.25}}
\put(64,16.75){\line(1,0){22.25}}
\put(112.25,16.75){\line(1,0){22.25}}
\multiput(25,34)(.0439453125,-.0336914063){512}{\line(1,0){.0439453125}}
\multiput(63.75,34)(.0439453125,-.0336914063){512}{\line(1,0){.0439453125}}
\multiput(112,34)(.0439453125,-.0336914063){512}{\line(1,0){.0439453125}}
\multiput(25.25,17)(.0439453125,.0336914063){512}{\line(1,0){.0439453125}}
\multiput(64,17)(.0439453125,.0336914063){512}{\line(1,0){.0439453125}}
\multiput(112.25,17)(.0439453125,.0336914063){512}{\line(1,0){.0439453125}}
\put(86.75,34){\line(1,0){25.5}}
\put(86.5,17.25){\line(1,0){25.5}}
\multiput(64,34)(.097082495,-.0337022133){497}{\line(1,0){.097082495}}
\multiput(86.5,16.75)(.0929672447,.0337186898){519}{\line(1,0){.0929672447}}
\multiput(86.5,34)(.0942460317,-.0337301587){504}{\line(1,0){.0942460317}}
\multiput(112,34)(-.0920303605,-.0336812144){527}{\line(-1,0){.0920303605}}
\qbezier(63.75,34.25)(101.5,49.25)(134.25,34.25)
\qbezier(64,16.5)(95.875,2.375)(134.25,16.75)
\put(5.25,34.5){\circle*{1.581}}
\put(5.5,15.5){\circle*{1.581}}
\put(24.75,16.75){\circle*{1.581}}
\put(25,33.5){\circle*{1.581}}
\put(47.75,33.5){\circle*{1.581}}
\put(47.75,16.25){\circle*{1.581}}
\put(64.5,16.25){\circle*{1.581}}
\put(63.75,34.25){\circle*{1.581}}
\put(87,34.25){\circle*{1.581}}
\put(112.5,34){\circle*{1.581}}
\put(135,33.75){\circle*{1.581}}
\put(134.5,17){\circle*{1.581}}
\put(112.25,16.5){\circle*{1.581}}
\put(86.5,17.25){\circle*{1.581}}
\put(5.25,38.5){\makebox(0,0)[cc]{$\tiny{0}$}}
\put(5.25,11){\makebox(0,0)[cc]{$\tiny{1}$}}
\put(23.5,38.5){\makebox(0,0)[cc]{$\tiny{00}$}}
\put(48.25,40.25){\makebox(0,0)[cc]{$\tiny{10}$}}
\put(23,12.25){\makebox(0,0)[cc]{$\tiny{01}$}}
\put(48.25,11.75){\makebox(0,0)[cc]{$\tiny{11}$}}
\put(64.5,39.5){\makebox(0,0)[cc]{$\tiny{000}$}}
\put(63.5,12.5){\makebox(0,0)[cc]{$\tiny{010}$}}
\put(85.5,37.5){\makebox(0,0)[cc]{ $\tiny{001}$}}
\put(85.5,12.5){\makebox(0,0)[cc]{$\tiny{011}$}}
\put(112,38){\makebox(0,0)[cc]{$\tiny{101}$}}
\put(136.25,38.25){\makebox(0,0)[cc]{$\tiny{100}$}}
\put(136.25,12.25){\makebox(0,0)[cc]{$\tiny{110}$}}
\put(111.75,13.25){\makebox(0,0)[cc]{$\tiny{111}$}}
\put(5.25,6.75){\makebox(0,0)[cc]{$\tiny AQ_{1}$}}
\put(35.25,6.75){\makebox(0,0)[cc]{$\tiny AQ_{2}$}}
\put(96.75,5.5){\makebox(0,0)[cc]{$\tiny AQ_{3}$}}
\put(70.5,.25){\makebox(0,0)[cc]
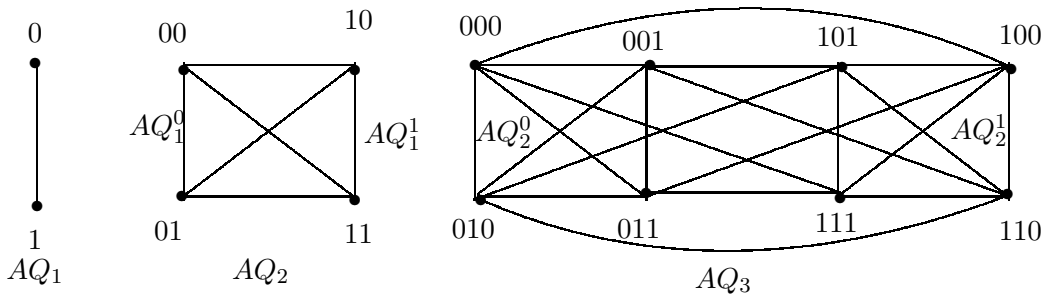
	{Fig. 1. The augmented cubes of dimension 1, 2 and 3 }}
\put(52.75,25){\makebox(0,0)[cc]{$\tiny AQ^1_{1}$}}
\put(21.5,26){\makebox(0,0)[cc]{$\tiny AQ^0_{1}$}}
\put(67.5,25.25){\makebox(0,0)[cc]{$\tiny AQ^0_{2}$}}
\put(130.75,25.5){\makebox(0,0)[cc]{$\tiny AQ^1_{2}$}}
\end{picture}
\end{center}
\vspace{.5cm}

The following lemma is helpful to visualize edges of $AQ_n$.\\
\begin{lemma} [\cite{cs}] Let $G$ be the simple graph with vertex set $V(G)= \{a_1a_2....a_n : a_i = 0$ or $ 1\}$ where two vertices $A= a_1a_2....a_n $ and $B = b_1b_2...b_n$ are joined iff there exists an integer $k:  1 \leq k \leq n$ such that either \\
	($1$) $a_k = \overline b_k$ and $a_i = b_i$, for every $i$, $i \neq k$, or\\
	($2$) $a_i = b_i$, for $1 \leq i \leq k-1$ and $a_i = \overline b_i$, for $k \leq i \leq n$.
\end{lemma}

\section {Construction of CISTs in augmented cubes }

We need the following result given in \cite{h1} by Hasunuma.\\

\begin{lemma} [\cite{h1} \label{lm1}] Let $k \geq 2$ be an integer, $T_1, T_2,.......,T_k$ are completely independent spanning trees in a graph $G$ if and only if they are edge-disjoint spanning trees of $G$ and for any $v \in V(G)$, there is at most one $T_i$ such that $d_{T_i}(v) > 1$. 
\end{lemma}

Pai and Chang\cite{p2} constructed two CISTs in several hypercube-variant networks, they proved the following result.

\begin{lemma} [\cite{p2} \label{lm2}] Let $G_n$ be the $n-$dimensional variant hypercube for $n \geq 4$ and suppose that $T_1$ and $T_2$ are two CISTs of $G_n$. For $i \in \{1,2\} $, let $\overline T_i$ be a spanning tree of $G_{n+1}$ constructed from $T^0_i$ and $T^1_i$ by adding an edge $\langle u_i, v_i \rangle \in E(G_{n+1})$ to connect two internal vertices $u_i \in V(T^0_i)$ and $v_i \in V(T^1_i)$. Then, $\overline T_1$ and $\overline T_2$ are two CISTs of $G_{n+1}$.
\end{lemma}

By using the same proof technique of above theorem one can state following corollary.

\begin{corollary} [ \label{lm3}]
	Let $G_n$ be the $n-$dimensional variant hypercube for $n \geq 4$ and suppose that $T_i$, for $1 \leq i \leq k $ $(k < n)$ be $k$ CISTs of $G_n$. Let $\overline T_i$ be a spanning tree of $G_{n+1}$ constructed from $T^0_i$ and $T^1_i$ by adding an edge $\langle u_i, v_i \rangle \in E(G_{n+1})$ to connect two internal vertices $u_i \in V(T^0_i)$ and $v_i \in V(T^1_i)$. Then, $\overline T_i$ are $k$ CISTs of $G_{n+1}$
\end{corollary}

Firstly, we list out $n-1$ CISTs in $AQ_n$ for $n = 3, 4$. See Fig.$2$
 and  Fig.$3$.\\
\begin{center}
 \unitlength 1.2mm 
 \linethickness{0.4pt}
 \ifx\plotpoint\undefined\newsavebox{\plotpoint}\fi 
 \begin{picture}(101.75,50.75)(0,0)
 \put(23.25,9.75){\circle*{3.202}}
 \put(22.851,32.351){\circle*{3.202}}
 \put(78.25,10.851){\circle*{3.202}}
 \put(77.851,33.452){\circle*{3.202}}
 \multiput(22.25,32.5)(.0434783,-.0326087){23}{\line(1,0){.0434783}}
 \put(8.5,36.75){\circle*{2.062}}
 \put(23.531,45.781){\circle*{2.062}}
 \put(39.281,36.531){\circle*{2.062}}
 \put(93.281,37.781){\circle*{2.062}}
 \put(63.281,39.031){\circle*{2.062}}
 \put(9.531,15.031){\circle*{2.062}}
 \put(37.031,19.281){\circle*{2.062}}
 \put(36.281,9.781){\circle*{2.062}}
 \put(64.281,19.531){\circle*{2.062}}
 \put(92.531,19.781){\circle*{2.062}}
 \put(94.281,10.781){\circle*{2.062}}
 \put(23.5,45.5){\line(0,-1){12.75}}
 \put(23.5,32.75){\line(4,1){16}}
 \multiput(8.5,36.5)(.12394958,-.033613445){119}{\line(1,0){.12394958}}
 \multiput(9,15)(.081871345,-.033625731){171}{\line(1,0){.081871345}}
 \put(36.25,9.75){\line(-1,0){12.75}}
 \multiput(63,38.75)(.097315436,-.033557047){149}{\line(1,0){.097315436}}
 \multiput(93.25,37.75)(-.128151261,-.033613445){119}{\line(-1,0){.128151261}}
 \multiput(64,19.25)(.0591836735,-.0336734694){245}{\line(1,0){.0591836735}}
 \put(93.75,10.75){\line(-1,0){15.25}}
 \put(43.5,9.5){\makebox(0,0)[cc]{$110(7)$}}
 \put(44.25,20){\makebox(0,0)[cc]{$011(4)$}}
 \put(3.75,17.5){\makebox(0,0)[cc]{$101(6)$}}
 \put(2.5,32.75){\makebox(0,0)[cc]{$010(3)$}}
 \put(29.5,28.25){\makebox(0,0)[cc]{$000(1)$}}
 \put(41.25,39.75){\makebox(0,0)[cc]{$111(8)$}}
 \put(15.75,7.25){\makebox(0,0)[cc]{$100(5)$}}
 \put(70.5,9.5){\makebox(0,0)[cc]{$111(8)$}}
 \put(101.75,10.25){\makebox(0,0)[cc]{$110(7)$}}
 \put(99,21.75){\makebox(0,0)[cc]{$100(5)$}}
 \put(63.5,21.75){\makebox(0,0)[cc]{$101(6)$}}
 \put(85.25,29.5){\makebox(0,0)[cc]{$011(4)$}}
 \put(98.25,39.5){\makebox(0,0)[cc]{$001(2)$}}
 \put(62.5,42.25){\makebox(0,0)[cc]{$010(3)$}}
 \put(52.75,1.5){\makebox(0,0)[cc]{Fig$. 2. $ Two CISTs in $AQ_3$}}
 \put(78.531,48.531){\circle*{2.062}}
 \put(77.5,50.75){\makebox(0,0)[cc]{$000(1)$}}
 \put(22.25,48.25){\makebox(0,0)[cc]{$001(2)$}}
 \multiput(77.25,33.25)(.125,.03125){8}{\line(1,0){.125}}
 \put(78.25,33.75){\line(0,1){15.25}}
 \put(22.75,31.5){\line(0,-1){22.25}}
 \put(77.5,32.5){\line(0,-1){21.75}}
 \multiput(23.25,10)(.0481818182,.0336363636){275}{\line(1,0){.0481818182}}
 \multiput(78,10.5)(.0518181818,.0336363636){275}{\line(1,0){.0518181818}}
 \end{picture}
 \end{center}

 \unitlength 1mm 
 \linethickness{0.4pt}
 \ifx\plotpoint\undefined\newsavebox{\plotpoint}\fi 
 \begin{picture}(151.5,120.25)(0,0)
 \put(20.75,10.5){\circle*{3.041}}
 \put(20.271,29.771){\circle*{3.041}}
 \put(20.771,48.771){\circle*{3.041}}
 \put(20.27,68.27){\circle*{3.041}}
 \put(20.77,87.02){\circle*{3.041}}
 \put(7.5,14.75){\circle*{1.581}}
 \put(33.541,16.791){\circle*{1.581}}
 \put(7.291,29.291){\circle*{1.581}}
 \put(33.082,35.582){\circle*{1.581}}
 \put(6.79,39.791){\circle*{1.581}}
 \put(7.041,52.791){\circle*{1.581}}
 \put(33.082,54.832){\circle*{1.581}}
 \put(6.541,72.041){\circle*{1.581}}
 \put(32.582,74.082){\circle*{1.581}}
 \put(33.082,93.082){\circle*{1.581}}
 \put(20.25,29.75){\line(0,-1){19.5}}
 \put(20.25,49){\line(0,-1){19.75}}
 \multiput(6.75,39.75)(.0460992908,-.0336879433){282}{\line(1,0){.0460992908}}
 \put(74.252,10.771){\circle*{3.041}}
 \put(73.772,30.041){\circle*{3.041}}
 \put(73.772,49.042){\circle*{3.041}}
 \put(73.772,68.541){\circle*{3.041}}
 \put(74.272,87.291){\circle*{3.041}}
 \put(61.252,10.771){\circle*{1.581}}
 \put(87.042,10.812){\circle*{1.581}}
 \put(60.793,34.562){\circle*{1.581}}
 \put(86.582,34.603){\circle*{1.581}}
 \put(60.541,18.811){\circle*{1.581}}
 \put(60.793,53.812){\circle*{1.581}}
 \put(86.582,53.853){\circle*{1.581}}
 \put(60.293,73.061){\circle*{1.581}}
 \put(86.082,73.102){\circle*{1.581}}
 \put(60.793,92.061){\circle*{1.581}}
 \put(86.582,92.102){\circle*{1.581}}
 \put(73.752,30.021){\line(0,-1){19.5}}
 \put(73.752,49.27){\line(0,-1){19.75}}
 \put(134.252,11.271){\circle*{3.041}}
 \put(133.772,30.541){\circle*{3.041}}
 \put(133.772,49.54){\circle*{3.041}}
 \put(133.772,69.04){\circle*{3.041}}
 \put(134.272,87.791){\circle*{3.041}}
 \put(121.252,11.271){\circle*{1.581}}
 \put(147.04,18.061){\circle*{1.581}}
 \put(120.793,36.81){\circle*{1.581}}
 \put(146.582,36.852){\circle*{1.581}}
 \put(120.54,20.311){\circle*{1.581}}
 \put(120.793,56.06){\circle*{1.581}}
 \put(146.582,56.101){\circle*{1.581}}
 \put(120.293,75.311){\circle*{1.581}}
 \put(120.793,94.311){\circle*{1.581}}
 \put(146.582,94.351){\circle*{1.581}}
 \put(133.752,30.52){\line(0,-1){19.5}}
 \put(133.752,49.77){\line(0,-1){19.75}}
 \multiput(73.5,48.75)(.03125,4.8125){8}{\line(0,1){4.8125}}
 \multiput(60.5,18.5)(.058823529,-.033613445){238}{\line(1,0){.058823529}}
 \put(134,87.75){\line(0,-1){39}}
 \put(134.54,99.041){\circle*{1.581}}
 \put(134,99){\line(0,-1){11}}
 \multiput(120.25,20.25)(.0535714286,-.0337301587){252}{\line(1,0){.0535714286}}
 \put(20,5.25){\makebox(0,0)[cc]{$0100 (5)$}}
 \put(5,16.75){\makebox(0,0)[cc]{$0111 (8)$}}
 \put(35.75,19.75){\makebox(0,0)[cc]{$1100 (13)$}}
 \put(28.5,28.5){\makebox(0,0)[cc]{$0101 (6)$}}
 \put(40.5,37.75){\makebox(0,0)[cc]{$0110 (7)$}}
 \put(4.75,24.75){\makebox(0,0)[cc]{$0010 (3)$}}
 \put(4.25,35){\makebox(0,0)[cc]{$0001 (2)$}}
 \put(28.75,47){\makebox(0,0)[cc]{$1010 (11)$}}
 \put(37.75,57){\makebox(0,0)[cc]{$1101 (14)$}}
 \put(5.25,55){\makebox(0,0)[cc]{$1110 (15)$}}
 \put(28.75,65.5){\makebox(0,0)[cc]{$1011 (12)$}}
 \put(38.25,76.25){\makebox(0,0)[cc]{$1001 (10)$}}
 \put(5,74.75){\makebox(0,0)[cc]{$0011 (4)$}}
 \multiput(20.75,86)(-.03125,-4.75){8}{\line(0,-1){4.75}}
 \put(28.75,86.25){\makebox(0,0)[cc]{$1000 (9)$}}
 \put(38.75,95.5){\makebox(0,0)[cc]{$1111 (16)$}}
 \put(74,6){\makebox(0,0)[cc]{$1111 (16)$}}
 \put(54.5,12.75){\makebox(0,0)[cc]{$1011 (12)$}}
 \put(61,20.5){\makebox(0,0)[cc]{$1101 (14)$}}
 \put(91,13.25){\makebox(0,0)[cc]{$0000 (1)$}}
 \put(81.75,27.5){\makebox(0,0)[cc]{$0111 (8)$}}
 \put(91.75,37){\makebox(0,0)[cc]{$0101 (6)$}}
 \put(60,36.75){\makebox(0,0)[cc]{$1000 (9)$}}
 \put(82.25,45.25){\makebox(0,0)[cc]{$0011 (4)$}}
 \put(91.75,56.25){\makebox(0,0)[cc]{$1100 (13)$}}
 \put(62.25,55.5){\makebox(0,0)[cc]{$0100 (5)$}}
 \put(81,64.5){\makebox(0,0)[cc]{$0001 (2)$}}
 \put(61.75,75.5){\makebox(0,0)[cc]{$0010 (3)$}}
 \put(90.25,75.75){\makebox(0,0)[cc]{$1110 (15)$}}
 \put(81.75,83.75){\makebox(0,0)[cc]{$1001 (10)$}}
 \put(61.75,95){\makebox(0,0)[cc]{$0110 (7)$}}
 \put(90.75,95.25){\makebox(0,0)[cc]{$1010 (11)$}}
 \put(133.75,5.5){\makebox(0,0)[cc]{$0110 (7)$}}
 \put(112.25,10.25){\makebox(0,0)[cc]{$0001 (2)$}}
 \put(112.25,20.25){\makebox(0,0)[cc]{$0100 (5)$}}
 \put(151.5,19.5){\makebox(0,0)[cc]{$0111 (8)$}}
 \put(141.5,25.75){\makebox(0,0)[cc]{$1110 (15)$}}
 \put(120.25,39.75){\makebox(0,0)[cc]{$1001 (10)$}}
 \put(151,39.75){\makebox(0,0)[cc]{$1111 (16)$}}
 \put(140.25,45){\makebox(0,0)[cc]{$1100 (13)$}}
 \put(120.5,58.5){\makebox(0,0)[cc]{$1000 (9)$}}
 \put(150.75,58.75){\makebox(0,0)[cc]{$1011 (12)$}}
 \put(143,68.75){\makebox(0,0)[cc]{$1101 (14)$}}
 \put(120,78){\makebox(0,0)[cc]{$0101 (6)$}}
 \put(140.5,82){\makebox(0,0)[cc]{$0010 (3)$}}
 \put(120.25,97.25){\makebox(0,0)[cc]{$0000 (1)$}}
 \put(134,102.25){\makebox(0,0)[cc]{$0011 (4)$}}
 \put(150.75,97.75){\makebox(0,0)[cc]{$1010 (11)$}}
 \put(73.25,.75){\makebox(0,0)[cc]{Fig$. 3.$ Three CISTs in $AQ_4$}}
 \put(21.25,87){\line(0,1){0}}
 \put(8,96.25){\makebox(0,0)[cc]{$0000 (1)$}}
 \put(7.5,29.5){\line(1,0){12.5}}
 \multiput(21,87)(.063172043,.033602151){186}{\line(1,0){.063172043}}
 \multiput(20.5,68.75)(.07885906,.033557047){149}{\line(1,0){.07885906}}
 \multiput(20.5,49.25)(.076219512,.033536585){164}{\line(1,0){.076219512}}
 \multiput(20,29.5)(.069892473,.033602151){186}{\line(1,0){.069892473}}
 \put(33,35.75){\line(0,1){0}}
 \multiput(20.75,11)(.070224719,.033707865){178}{\line(1,0){.070224719}}
 \multiput(6.5,72.25)(.103174603,-.033730159){126}{\line(1,0){.103174603}}
 \multiput(7,53)(.122596154,-.033653846){104}{\line(1,0){.122596154}}
 \multiput(7.5,14.75)(.109243697,-.033613445){119}{\line(1,0){.109243697}}
 \multiput(60.75,92)(.092198582,-.033687943){141}{\line(1,0){.092198582}}
 \multiput(74,87.5)(.088652482,.033687943){141}{\line(1,0){.088652482}}
 \multiput(60,73)(.104477612,-.03358209){134}{\line(1,0){.104477612}}
 \multiput(85.75,73)(-.089552239,-.03358209){134}{\line(-1,0){.089552239}}
 \multiput(60.5,53.5)(.08557047,-.033557047){149}{\line(1,0){.08557047}}
 \multiput(86.5,53.75)(-.090425532,-.033687943){141}{\line(-1,0){.090425532}}
 \put(60.25,34.25){\line(3,-1){13.5}}
 \multiput(86.25,34.5)(-.097014925,-.03358209){134}{\line(-1,0){.097014925}}
 \multiput(120.75,94.25)(.065920398,-.03358209){201}{\line(1,0){.065920398}}
 \multiput(146.5,94.5)(-.058894231,-.033653846){208}{\line(-1,0){.058894231}}
 \multiput(120.25,75.25)(.07253886,-.033678756){193}{\line(1,0){.07253886}}
 \multiput(120.75,56)(.06865285,-.033678756){193}{\line(1,0){.06865285}}
 \multiput(146.25,56.25)(-.063432836,-.03358209){201}{\line(-1,0){.063432836}}
 \multiput(120.25,36.75)(.06840796,-.03358209){201}{\line(1,0){.06840796}}
 \multiput(146.5,36.75)(-.066062176,-.033678756){193}{\line(-1,0){.066062176}}
 \put(121.25,11.25){\line(1,0){12.75}}
 \multiput(146.75,18)(-.063471503,-.033678756){193}{\line(-1,0){.063471503}}
 \put(61,10.75){\line(1,0){13.25}}
 \put(86.75,11){\line(-1,0){12.25}}
 \put(6.791,94.04){\circle*{1.581}}
 \multiput(6.5,94.25)(.066105769,-.033653846){208}{\line(1,0){.066105769}}
 \end{picture}

\vspace{2cm}
 {\bf Note} : To make things readable, we denote vertices of $AQ_5$ by using numbers $1,2,.....,32$. For example the vertex $00000$ will be denoted by $1$ and sometime will
 be written as $00000(1)$. Also, we will use short forms to denote edges of $AQ_5$, for example the edge $\langle 00000(1), 00001(2) \rangle$ will be denoted by $\langle 1, 2 \rangle$. For every $i$, $InV(T_i)$ denotes set of internal vertices of the tree $T_i$, for $1 \leq i \leq 4$. Vertex $v \in V(T_i)$ is called internal if $d_{T_i}(v) \geq 2$. \\
 Let $ V(AQ_5) =  \{  00000(1), 00001(2), 00010(3), 00011(4), 00100(5), 00101(6), 00110(7), 00111(8),$ 
 
 $\hspace{2.3cm} 01000(9), 01001(10), 01010(11), 01011(12), 01100(13), 01101(14), 01110(15), $
 
 $\hspace{2.3cm} 01111(16), 10000(17), 10001(18), 10010(19), 10011(20), 10100(21), 10101(22),
 $
 
 $ \hspace{2.3cm} 10110(23), 10111(24), 11000(25), 11001(26), 11010(27), 11011(28), 11100(29),$
 
 $\hspace{2.3cm}  11101(30), 11110(31), 11111(32) \}.$\\
 
 Now for $n = 5$, we construct four trees $T_1, T_2, T_3$ and $T_4$ as shown in Fig.$4(a)$, Fig.$4(b)$, Fig.$4(c)$ and Fig.$4(d)$ respectively.\\
\[ 
\unitlength 1mm 
\linethickness{0.4pt}
\ifx\plotpoint\undefined\newsavebox{\plotpoint}\fi 
\begin{picture}(100.798,110.397)(0,0)
\put(33.706,13.409){\circle*{2.268}}
\put(33.48,29.511){\circle*{2.268}}
\put(33.48,61.943){\circle*{2.268}}
\put(33.706,78.954){\circle*{2.268}}
\put(33.48,95.057){\circle*{2.268}}
\put(33.706,111.386){\circle*{2.268}}
\put(33.253,14.316){\line(0,1){15.422}}
\put(33.253,29.965){\line(0,1){16.557}}
\put(33.253,46.522){\line(0,1){16.102}}
\put(33.253,62.624){\line(0,1){17.01}}
\put(33.253,79.634){\line(0,1){16.103}}
\put(33.253,95.737){\line(0,1){16.556}}
\put(50.716,11.594){\circle*{1.434}}
\put(50.3,19.115){\circle*{1.434}}
\put(15.109,11.404){\circle*{1.434}}
\put(14.692,18.926){\circle*{1.434}}
\put(50.489,27.017){\circle*{1.434}}
\put(50.072,34.538){\circle*{1.434}}
\put(14.881,26.826){\circle*{1.434}}
\put(14.465,34.348){\circle*{1.434}}
\put(51.434,46.331){\circle*{1.434}}
\put(13.522,57.671){\circle*{1.434}}
\put(13.105,65.192){\circle*{1.434}}
\put(52.34,62.434){\circle*{1.434}}
\put(51.623,78.727){\circle*{2.268}}
\put(33.933,78.954){\line(1,0){17.463}}
\put(13.066,75.588){\circle*{1.434}}
\put(12.65,83.11){\circle*{1.434}}
\put(66.138,72.186){\circle*{1.434}}
\put(65.268,79.028){\circle*{1.434}}
\put(65.91,87.608){\circle*{1.434}}
\put(49.129,91.882){\circle*{1.434}}
\put(48.712,99.402){\circle*{1.434}}
\put(13.522,91.691){\circle*{1.434}}
\put(13.104,99.213){\circle*{1.434}}
\put(51.433,118.68){\circle*{1.434}}
\put(51.434,111.196){\circle*{1.434}}
\put(13.785,111.877){\circle*{1.434}}
\multiput(33.933,13.635)(.30660741,-.0336){54}{\line(1,0){.30660741}}
\put(33.706,13.862){\line(3,1){16.33}}
\multiput(33.706,29.738)(.20439506,-.0336){81}{\line(1,0){.20439506}}
\multiput(33.48,29.965)(.116591549,.033543662){142}{\line(1,0){.116591549}}
\put(33.933,46.522){\line(1,0){17.237}}
\put(33.48,62.398){\line(1,0){18.824}}
\multiput(33.933,95.283)(.148205941,-.033679208){101}{\line(1,0){.148205941}}
\multiput(33.706,95.51)(.132132174,.033523478){115}{\line(1,0){.132132174}}
\put(33.933,111.613){\line(1,0){17.237}}
\multiput(33.706,112.066)(.086455446,.033683168){202}{\line(1,0){.086455446}}
\multiput(33.48,95.283)(-.170571901,.033738843){121}{\line(-1,0){.170571901}}
\multiput(33.253,95.057)(-.197607921,-.033687129){101}{\line(-1,0){.197607921}}
\multiput(33.48,78.954)(-.185384348,.033530435){115}{\line(-1,0){.185384348}}
\multiput(33.48,78.727)(-.206590099,-.033687129){101}{\line(-1,0){.206590099}}
\multiput(32.799,62.17)(-.210088421,.033423158){95}{\line(-1,0){.210088421}}
\multiput(33.48,61.717)(-.1577,-.0336625){128}{\line(-1,0){.1577}}
\multiput(33.253,29.965)(-.142797037,.0336){135}{\line(-1,0){.142797037}}
\multiput(33.253,29.738)(-.20618182,-.03350909){88}{\line(-1,0){.20618182}}
\multiput(33.48,13.409)(-.112728994,.033547929){169}{\line(-1,0){.112728994}}
\multiput(33.706,13.409)(-.30859016,-.03346885){61}{\line(-1,0){.30859016}}
\multiput(51.397,79.407)(.055815625,.033665625){256}{\line(1,0){.055815625}}
\put(51.85,78.5){\line(2,-1){14.062}}
\put(52.078,78.954){\line(1,0){12.927}}
\put(33.026,8.192){\makebox(0,0)[cc]{$11110(31)$}}
\put(60.015,11.367){\makebox(0,0)[cc]{$11111(32)$}}
\put(59.201,19.306){\makebox(0,0)[cc]{$11101(30)$}}
\put(6.171,10.687){\makebox(0,0)[cc]{$11100(29)$}}
\put(6.398,19.078){\makebox(0,0)[cc]{$10110(23)$}}
\put(25.676,24.749){\makebox(0,0)[cc]{$11010(27)$}}
\put(60.108,27.017){\makebox(0,0)[cc]{$10101(22)$}}
\put(59.108,34.954){\makebox(0,0)[cc]{$10010(19)$}}
\put(61.15,46.294){\makebox(0,0)[cc]{$01000(9)$}}
\put(12.614,36.542){\makebox(0,0)[cc]{$01010(11)$}}
\put(6.171,26.563){\makebox(0,0)[cc]{$11001(26)$}}
\put(61.922,62.398){\makebox(0,0)[cc]{$10111(24)$}}
\put(5.398,57.862){\makebox(0,0)[cc]{$10011(20)$}}
\put(10.8,66.933){\makebox(0,0)[cc]{$10001(18)$}}
\put(24.542,46.522){\makebox(0,0)[cc]{$11000(25)$}}
\put(41.964,58.315){\makebox(0,0)[cc]{$10000(17)$}}
\put(26.449,73.057){\makebox(0,0)[cc]{$00000(1)$ }}
\put(47.994,73.51){\makebox(0,0)[cc]{$01111(16)$}}
\put(76.798,71.47){\makebox(0,0)[cc]{$01100(13)$}}
\put(74.758,78.954){\makebox(0,0)[cc]{$01101(14)$}}
\put(76.345,87.572){\makebox(0,0)[cc]{$01110(15)$}}
\put(10.573,77.82){\makebox(0,0)[cc]{$00010(3)$ }}
\put(10.346,84.397){\makebox(0,0)[cc]{$00001(2)$}}
\put(39.15,90.52){\makebox(0,0)[cc]{$00100(5)$}}
\put(58.52,92.562){\makebox(0,0)[cc]{$10100(21)$}}
\put(59.067,100.046){\makebox(0,0)[cc]{$11011(28)$}}
\put(11.026,93.469){\makebox(0,0)[cc]{$00011(4)$ }}
\put(9.439,101.18){\makebox(0,0)[cc]{$01011(12)$}}
\put(26.902,114.077){\makebox(0,0)[cc]{$00110(7)$}}
\put(11.026,108.21){\makebox(0,0)[cc]{$00101(6)$}}
\put(59.334,110.706){\makebox(0,0)[cc]{$00111(8)$}}
\put(60.428,118.643){\makebox(0,0)[cc]{$01001(10)$}}
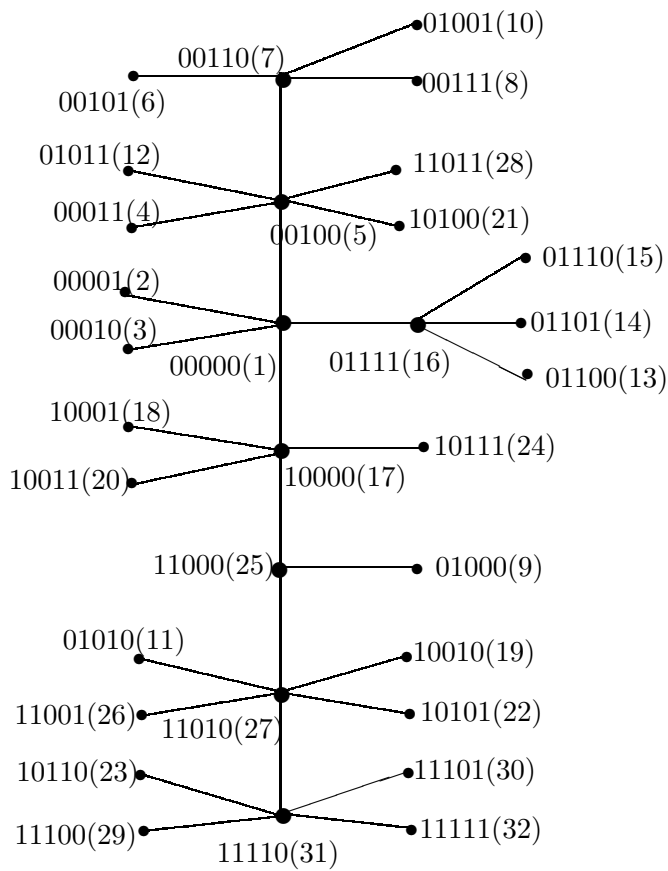
\put(32.572,1.615){\makebox(0,0)[cc]{Fig. 4(a). Spanning Tree $T_1$ in $AQ_5$}}
\put(33.253,111.84){\line(-1,0){19.505}}
\put(33.134,46.134){\circle*{2.268}}
\end{picture}
\]
 
\unitlength 1mm 
\linethickness{0.4pt}
\ifx\plotpoint\undefined\newsavebox{\plotpoint}\fi 
\begin{picture}(81.25,111.665)(0,0)
\put(51.65,9.4){\circle*{2.332}}
\put(51.816,23.966){\circle*{2.332}}
\put(51.816,37.366){\circle*{2.332}}
\put(51.983,51.933){\circle*{2.332}}
\put(52.216,67.566){\circle*{2.332}}
\put(52.383,82.133){\circle*{2.332}}
\put(52.383,95.533){\circle*{2.332}}
\put(73.65,8.2){\circle*{1.265}}
\put(73.683,15.233){\circle*{1.265}}
\put(30.283,9.233){\circle*{1.265}}
\put(30.316,16.266){\circle*{1.265}}
\put(31.416,36.766){\circle*{2.332}}
\put(15.683,30.832){\circle*{1.265}}
\put(15.716,37.066){\circle*{1.265}}
\put(15.882,44.433){\circle*{1.265}}
\put(72.883,31.833){\circle*{1.265}}
\put(72.916,38.866){\circle*{1.265}}
\put(73.082,45.433){\circle*{1.265}}
\put(73.082,23.833){\circle*{1.265}}
\put(32.482,50.433){\circle*{1.265}}
\put(32.649,56.998){\circle*{1.265}}
\put(72.882,55.433){\circle*{1.265}}
\put(32.683,67.632){\circle*{1.265}}
\put(73.082,67.832){\circle*{1.265}}
\put(73.45,77.633){\circle*{1.265}}
\put(73.483,84.666){\circle*{1.265}}
\put(32.283,78.466){\circle*{1.265}}
\put(32.316,85.498){\circle*{1.265}}
\put(32.283,95.832){\circle*{1.265}}
\put(32.083,105.033){\circle*{1.265}}
\put(73.283,95.832){\circle*{1.265}}
\put(73.083,105.033){\circle*{1.265}}
\put(51.85,23.8){\line(0,-1){14.2}}
\put(51.85,37.2){\line(0,-1){13.4}}
\put(51.65,52){\line(0,-1){14.2}}
\put(52.25,67.4){\line(0,-1){16.6}}
\put(52.45,82.4){\line(0,-1){14.6}}
\put(52.25,95.4){\line(0,-1){13.2}}
\put(31.45,36.6){\line(0,1){.2}}
\put(31.05,37.2){\line(1,0){20.4}}
\multiput(51.65,9.6)(.127906977,.03372093){172}{\line(1,0){.127906977}}
\multiput(52.05,9.6)(.51428571,-.03333333){42}{\line(1,0){.51428571}}
\multiput(51.25,9.6)(-.104950495,.033663366){202}{\line(-1,0){.104950495}}
\put(51.65,24){\line(1,0){21}}
\multiput(31.45,36.8)(-.068103448,.03362069){232}{\line(-1,0){.068103448}}
\put(31.05,37){\line(-1,0){15.8}}
\multiput(31.05,36.4)(-.090697674,-.03372093){172}{\line(-1,0){.090697674}}
\multiput(52.25,37.6)(.11744186,-.03372093){172}{\line(1,0){.11744186}}
\multiput(51.85,37.6)(.4952381,.03333333){42}{\line(1,0){.4952381}}
\multiput(51.45,37.6)(.092241379,.03362069){232}{\line(1,0){.092241379}}
\multiput(51.65,52)(-.12885906,.033557047){149}{\line(-1,0){.12885906}}
\put(51.65,51.6){\line(0,1){.4}}
\multiput(51.65,52)(-.40833333,-.03333333){48}{\line(-1,0){.40833333}}
\multiput(51.85,51.8)(.184070796,.033628319){113}{\line(1,0){.184070796}}
\put(52.05,67.6){\line(1,0){20.8}}
\put(51.85,67.4){\line(-1,0){19.4}}
\multiput(52.05,82.4)(.27179487,.03333333){78}{\line(1,0){.27179487}}
\multiput(52.45,82.2)(.154744526,-.033576642){137}{\line(1,0){.154744526}}
\multiput(52.65,82.2)(-.214736842,.033684211){95}{\line(-1,0){.214736842}}
\multiput(52.25,82)(-.190654206,-.03364486){107}{\line(-1,0){.190654206}}
\multiput(52.45,95.6)(.0731182796,.0336917563){279}{\line(1,0){.0731182796}}
\put(52.65,95.8){\line(1,0){20.6}}
\put(52.45,96){\line(-1,0){20.4}}
\multiput(52.05,95.8)(-.0724014337,.0336917563){279}{\line(-1,0){.0724014337}}
\put(51.25,4.8){\makebox(0,0)[cc]{$11101(30)$}}
\put(84.05,7.8){\makebox(0,0)[cc]{$11111(32)$}}
\put(83.85,16){\makebox(0,0)[cc]{$10010(19)$}}
\put(18.85,7.8){\makebox(0,0)[cc]{$11010(27)$}}
\put(21.25,15.6){\makebox(0,0)[cc]{$01101(14)$}}
\put(82.25,24){\makebox(0,0)[cc]{$11110(31)$}}
\put(43.25,34.6){\makebox(0,0)[cc]{$10110(23)$}}
\put(81.85,32){\makebox(0,0)[cc]{$10101(22)$}}
\put(82.25,38.6){\makebox(0,0)[cc]{$10001(18)$}}
\put(82.85,45.8){\makebox(0,0)[cc]{$00110(7)$}}
\put(33.25,40.2){\makebox(0,0)[cc]{$10100(21)$}}
\put(6.45,30){\makebox(0,0)[cc]{$10000(17)$}}
\put(6.45,36.8){\makebox(0,0)[cc]{$10011(20)$}}
\put(6.25,44.4){\makebox(0,0)[cc]{$11011(28)$}}
\put(62.05,50){\makebox(0,0)[cc]{$01001(10)$}}
\put(82.05,55.4){\makebox(0,0)[cc]{$01110(15)$}}
\put(23.65,50.2){\makebox(0,0)[cc]{$01000(9)$}}
\put(23.05,56.8){\makebox(0,0)[cc]{$01010(11)$}}
\put(60.25,65){\makebox(0,0)[cc]{$00001(2)$}}
\put(82.45,68){\makebox(0,0)[cc]{$00101(6)$}}
\put(24.65,66.8){\makebox(0,0)[cc]{$00010(3)$}}
\put(44.85,77.2){\makebox(0,0)[cc]{$00011(4)$}}
\put(83.25,84.8){\makebox(0,0)[cc]{$01011(12)$}}
\put(83.45,77.6){\makebox(0,0)[cc]{$01100(13)$}}
\put(23.65,85.4){\makebox(0,0)[cc]{$00000(1)$}}
\put(23.65,78.2){\makebox(0,0)[cc]{$11100(29)$}}
\put(44.45,93.2){\makebox(0,0)[cc]{$00111(8)$}}
\put(82.25,105.6){\makebox(0,0)[cc]{$00100(5)$}}
\put(82.25,96){\makebox(0,0)[cc]{$01111(16)$}}
\put(23.25,105){\makebox(0,0)[cc]{$10111(24)$}}
\put(23.25,95.4){\makebox(0,0)[cc]{$11000(25)$}}
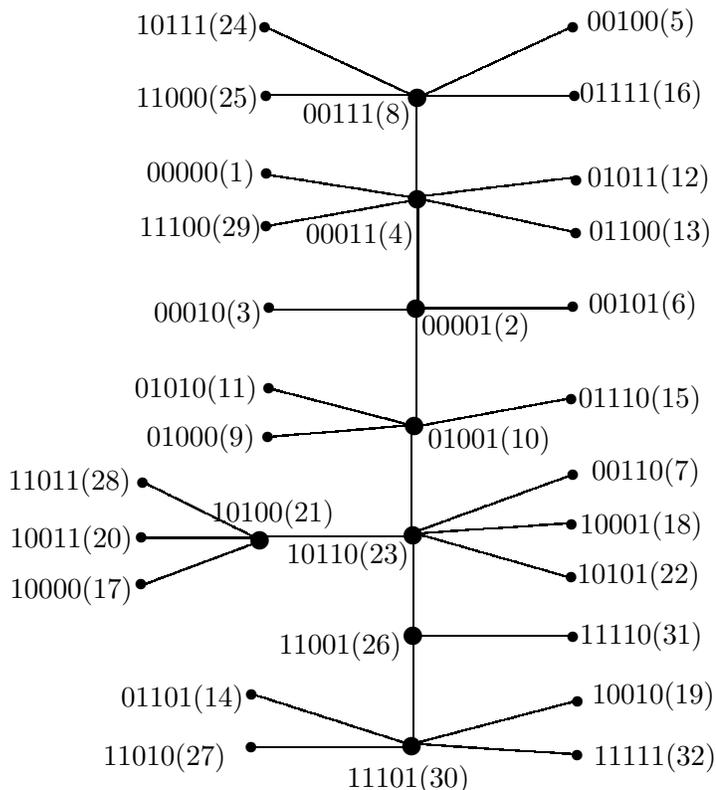
\put(50.65,1.0){\makebox(0,0)[cc]{Fig. 4(b). Spanning Tree $T_2$ in $AQ_5$}}
\put(42.25,22.6){\makebox(0,0)[cc]{$11001(26)$}}
\put(29.85,9.2){\line(1,0){21.6}}
\end{picture}

\[ 
\unitlength 1mm 
\linethickness{0.4pt}
\ifx\plotpoint\undefined\newsavebox{\plotpoint}\fi 
\begin{picture}(92.73,115.739)(0,0)
\put(48.09,10.92){\circle*{2.036}}
\put(47.668,31.558){\circle*{2.036}}
\put(47.91,31.08){\line(0,-1){20.34}}
\put(47.73,54.418){\circle*{2.036}}
\put(47.308,75.056){\circle*{2.036}}
\put(47.55,74.578){\line(0,-1){20.34}}
\put(47.488,94.738){\circle*{2.036}}
\put(47.55,94.44){\line(0,-1){19.08}}
\put(47.73,54.12){\line(0,-1){22.5}}
\put(28.408,54.418){\circle*{2.036}}
\put(28.228,74.938){\circle*{2.036}}
\put(68.008,74.938){\circle*{2.036}}
\put(29.19,14.7){\circle*{1.138}}
\put(29.22,8.07){\circle*{1.138}}
\put(67.74,10.95){\circle*{1.138}}
\put(67.38,31.47){\circle*{1.138}}
\put(66.66,55.379){\circle*{1.138}}
\put(66.689,48.749){\circle*{1.138}}
\put(66.84,61.17){\circle*{1.138}}
\put(11.759,53.579){\circle*{1.138}}
\put(11.788,46.949){\circle*{1.138}}
\put(11.939,59.37){\circle*{1.138}}
\put(12.12,76.439){\circle*{1.138}}
\put(12.149,69.809){\circle*{1.138}}
\put(12.3,82.23){\circle*{1.138}}
\put(12.66,87.63){\circle*{1.138}}
\put(60.54,84.75){\circle*{1.138}}
\put(85.38,75.719){\circle*{1.138}}
\put(85.409,69.089){\circle*{1.138}}
\put(85.56,81.51){\circle*{1.138}}
\put(31.38,100.379){\circle*{1.138}}
\put(31.409,93.749){\circle*{1.138}}
\put(31.56,106.17){\circle*{1.138}}
\put(67.02,102.539){\circle*{1.138}}
\put(67.049,95.909){\circle*{1.138}}
\multiput(29.19,14.88)(.157118644,-.033559322){118}{\line(1,0){.157118644}}
\multiput(29.19,7.86)(.21767442,.03348837){86}{\line(1,0){.21767442}}
\put(48.27,11.1){\line(1,0){19.26}}
\put(47.73,31.44){\line(1,0){19.8}}
\put(47.73,31.44){\line(-1,0){19.08}}
\multiput(47.73,54.3)(.116273292,-.033540373){161}{\line(1,0){.116273292}}
\multiput(47.37,54.3)(.8672727,.0327273){22}{\line(1,0){.8672727}}
\multiput(66.63,60.96)(-.104835165,-.033626374){182}{\line(-1,0){.104835165}}
\put(28.47,54.3){\line(1,0){18.9}}
\multiput(11.91,59.16)(.122238806,-.03358209){134}{\line(1,0){.122238806}}
\multiput(11.73,53.58)(.7363636,.0327273){22}{\line(1,0){.7363636}}
\multiput(11.91,47.1)(.080689655,.033694581){203}{\line(1,0){.080689655}}
\multiput(12.27,69.78)(.114206897,.033517241){145}{\line(1,0){.114206897}}
\multiput(12.09,76.26)(.43105263,-.03315789){38}{\line(1,0){.43105263}}
\multiput(12.27,82.2)(.080095694,-.033588517){209}{\line(1,0){.080095694}}
\multiput(12.63,87.6)(.0446280992,-.0337190083){363}{\line(1,0){.0446280992}}
\put(28.47,75){\line(1,0){19.26}}
\put(47.55,75){\line(1,0){20.52}}
\multiput(68.07,75.18)(.092406417,.03368984){187}{\line(1,0){.092406417}}
\multiput(67.89,75.36)(1.538182,.032727){11}{\line(1,0){1.538182}}
\multiput(67.71,75.18)(.095934066,-.033626374){182}{\line(1,0){.095934066}}
\multiput(47.37,75.18)(.045971223,.0336690647){278}{\line(1,0){.045971223}}
\multiput(47.73,94.62)(.08373913,.033652174){230}{\line(1,0){.08373913}}
\multiput(47.55,94.8)(.72,.0333333){27}{\line(1,0){.72}}
\multiput(31.53,106.32)(.0483987915,-.0337160121){331}{\line(1,0){.0483987915}}
\multiput(31.35,100.2)(.105677419,-.033677419){155}{\line(1,0){.105677419}}
\multiput(31.17,93.9)(.7281818,.0327273){22}{\line(1,0){.7281818}}
\put(49.55,7.86){\makebox(0,0)[cc]{$10010(19)$}}
\put(76.55,11.1){\makebox(0,0)[cc]{$10101(22)$}}
\put(20.07,14.52){\makebox(0,0)[cc]{$10110(23)$}}
\put(20.53,7.5){\makebox(0,0)[cc]{$10000(17)$}}
\put(40.13,28.28){\makebox(0,0)[cc]{$10011(20)$}}
\put(76.47,30.36){\makebox(0,0)[cc]{$10111(24)$}}
\put(20.25,30.18){\makebox(0,0)[cc]{$00011(4)$}}
\put(39.51,50.96){\makebox(0,0)[cc]{$11100(29)$}}
\put(76.21,61.14){\makebox(0,0)[cc]{$10100(21)$}}
\put(76.57,55.38){\makebox(0,0)[cc]{$11000(25)$}}
\put(76.57,48.18){\makebox(0,0)[cc]{$11101(30)$}}
\put(28.83,58.08){\makebox(0,0)[cc]{$11011(28)$}}
\put(2.61,59.52){\makebox(0,0)[cc]{$11001(26)$}}
\put(2.97,53.4){\makebox(0,0)[cc]{$11010(27)$}}
\put(2.25,46.38){\makebox(0,0)[cc]{$11111(32)$}}
\put(39.33,72.48){\makebox(0,0)[cc]{$01100(13)$}}
\put(69.27,85.26){\makebox(0,0)[cc]{$00100(5)$}}
\put(65.01,70.94){\makebox(0,0)[cc]{$01101(14)$}}
\put(94.73,81.66){\makebox(0,0)[cc]{$00010(3)$}}
\put(94.65,75.36){\makebox(0,0)[cc]{$00101(6)$}}
\put(94.93,68.7){\makebox(0,0)[cc]{$01001(10)$}}
\put(34.27,78.42){\makebox(0,0)[cc]{$01110(15)$}}
\put(3.71,87.96){\makebox(0,0)[cc]{$00001(2)$ }}
\put(3.61,82.02){\makebox(0,0)[cc]{$00110(7)$}}
\put(3.61,75.9){\makebox(0,0)[cc]{$10001(18)$}}
\put(3.97,69.42){\makebox(0,0)[cc]{$11110(31)$}}
\put(56.77,92.1){\makebox(0,0)[cc]{$01000(9)$}}
\put(76.93,102.72){\makebox(0,0)[cc]{$00000(1)$ }}
\put(76.29,95.7){\makebox(0,0)[cc]{$00111(8)$}}
\put(23.59,106.32){\makebox(0,0)[cc]{$01010(11)$}}
\put(23.41,99.84){\makebox(0,0)[cc]{$01011(12)$}}
\put(23.23,93.54){\makebox(0,0)[cc]{$01111(16)$}}
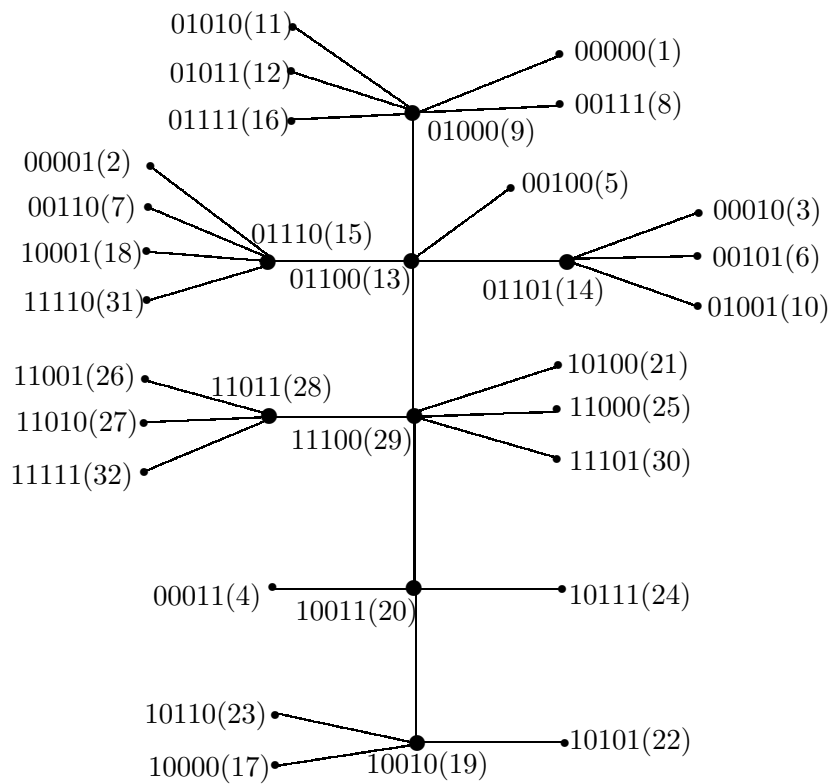
\put(47.01,2.46){\makebox(0,0)[cc]{Fig. 4(c). Spanning Tree $T_3$ in $AQ_5$}}
\put(28.86,31.65){\circle*{1.138}}
\end{picture}
\]

\[
\unitlength 1mm 
\linethickness{0.4pt}
\ifx\plotpoint\undefined\newsavebox{\plotpoint}\fi 
\begin{picture}(100.7,110.48)(0,0)
\put(56.42,11.435){\circle*{2.305}}
\put(56.133,32.388){\circle*{2.305}}
\put(56.6,53.808){\circle*{2.305}}
\put(56.313,74.76){\circle*{2.305}}
\put(55.593,95.028){\circle*{2.305}}
\put(35.613,53.988){\circle*{2.305}}
\put(35.433,74.148){\circle*{2.305}}
\put(78.633,74.688){\circle*{2.305}}
\put(72.62,18.095){\circle*{1.61}}
\put(72.885,7.92){\circle*{1.61}}
\put(37.065,10.62){\circle*{1.61}}
\put(72.885,40.235){\circle*{1.61}}
\put(73.15,30.06){\circle*{1.61}}
\put(35.445,32.04){\circle*{1.61}}
\put(72.705,62.015){\circle*{1.61}}
\put(72.97,51.84){\circle*{1.61}}
\put(22.305,39.06){\circle*{1.61}}
\put(16.005,46.62){\circle*{1.61}}
\put(12.945,53.28){\circle*{1.61}}
\put(13.665,61.02){\circle*{1.61}}
\put(19.425,67.14){\circle*{1.61}}
\put(94.125,85.86){\circle*{1.61}}
\put(95.565,74.88){\circle*{1.61}}
\put(93.405,66.24){\circle*{1.61}}
\put(77.745,104.675){\circle*{1.61}}
\put(77.83,95.22){\circle*{1.61}}
\put(38.865,104.675){\circle*{1.61}}
\put(39.13,94.5){\circle*{1.61}}
\put(17.265,78.3){\circle*{1.61}}
\put(25.725,87.3){\circle*{1.61}}
\put(56.24,11.615){\line(0,1){20.88}}
\put(56.24,53.735){\line(0,-1){21.42}}
\put(56.24,74.435){\line(0,-1){20.7}}
\put(56.24,74.075){\line(0,1){0}}
\put(78.74,74.615){\line(-1,0){22.68}}
\multiput(56.24,53.915)(.0638247012,.0337051793){251}{\line(1,0){.0638247012}}
\multiput(56.42,53.735)(.252,-.03323077){65}{\line(1,0){.252}}
\multiput(78.56,74.435)(.0588047809,-.0337051793){251}{\line(1,0){.0588047809}}
\put(78.2,74.975){\line(1,0){17.46}}
\multiput(78.56,74.975)(.0469325153,.0336809816){326}{\line(1,0){.0469325153}}
\put(55.52,95.135){\line(1,0){22.5}}
\multiput(54.98,95.495)(.0784775087,.033633218){289}{\line(1,0){.0784775087}}
\put(39.32,94.415){\line(1,0){16.38}}
\multiput(38.6,105.035)(.0569387755,-.0336734694){294}{\line(1,0){.0569387755}}
\multiput(25.64,87.215)(.033633218,-.0435986159){289}{\line(0,-1){.0435986159}}
\multiput(17,78.395)(.157699115,-.033451327){113}{\line(1,0){.157699115}}
\multiput(19.34,67.055)(.0422691293,-.0337203166){379}{\line(1,0){.0422691293}}
\multiput(13.4,60.755)(.107655502,-.033588517){209}{\line(1,0){.107655502}}
\put(12.68,53.195){\line(1,0){22.68}}
\multiput(15.92,46.895)(.108791209,.033626374){182}{\line(1,0){.108791209}}
\multiput(22.22,39.155)(.0336945813,.0345812808){406}{\line(0,1){.0345812808}}
\put(35.54,53.915){\line(1,0){20.52}}
\put(35.18,31.775){\line(1,0){20.34}}
\put(36.98,10.715){\line(1,0){19.26}}
\multiput(56.78,11.615)(.148037383,-.03364486){107}{\line(1,0){.148037383}}
\multiput(56.24,11.435)(.081818182,.033636364){198}{\line(1,0){.081818182}}
\multiput(56.06,32.675)(.228,-.0336){75}{\line(1,0){.228}}
\multiput(56.24,32.315)(.0687804878,.0336585366){246}{\line(1,0){.0687804878}}
\put(35.9,74.615){\line(1,0){20.34}}
\put(55.88,94.595){\line(0,-1){19.62}}
\put(54.44,6.935){\makebox(0,0)[cc]{$11111(32)$}}
\put(82.1,7.475){\makebox(0,0)[cc]{$11100(29)$}}
\put(82.92,18.275){\makebox(0,0)[cc]{$10000(17)$}}
\put(28.78,10.175){\makebox(0,0)[cc]{$00000(1)$}}
\put(47.84,28.355){\makebox(0,0)[cc]{$10111(24)$}}
\put(82.56,40.595){\makebox(0,0)[cc]{$10110(23)$}}
\put(82.46,30.155){\makebox(0,0)[cc]{$11000(25)$}}
\put(27.5,30.695){\makebox(0,0)[cc]{$01000(9)$}}
\put(82.64,51.935){\makebox(0,0)[cc]{$10100(21)$}}
\put(82.38,62.555){\makebox(0,0)[cc]{$11101(30)$}}
\put(7.86,67.595){\makebox(0,0)[cc]{$00001(2)$}}
\put(3.82,60.935){\makebox(0,0)[cc]{$10010(19)$}}
\put(3.74,53.195){\makebox(0,0)[cc]{$10011(20)$}}
\put(5.26,45.635){\makebox(0,0)[cc]{$11001(26)$}}
\put(11.38,38.435){\makebox(0,0)[cc]{$11110(31)$}}
\put(65.94,70.915){\makebox(0,0)[cc]{$01010(11)$}}
\put(41.7,76.855){\makebox(0,0)[cc]{$00010(3)$}}
\put(16.16,87.755){\makebox(0,0)[cc]{$00011(4)$}}
\put(8.34,78.395){\makebox(0,0)[cc]{$00110(7)$}}
\put(74.42,77.855){\makebox(0,0)[cc]{$00101(6)$}}
\put(102.8,85.955){\makebox(0,0)[cc]{$00100(5)$}}
\put(104.7,74.975){\makebox(0,0)[cc]{$00111(8)$}}
\put(102.72,66.155){\makebox(0,0)[cc]{$11010(27)$}}
\put(65.64,90.815){\makebox(0,0)[cc]{$01011(12)$}}
\put(29.48,104.315){\makebox(0,0)[cc]{$01001(10)$}}
\put(30.02,96.235){\makebox(0,0)[cc]{$01100(13)$}}
\put(87.96,105.035){\makebox(0,0)[cc]{$01111(16)$}}
\put(87.96,94.595){\makebox(0,0)[cc]{$11011(28)$}}
\put(48.945,84.6){\circle*{1.61}}
\put(48.585,64.26){\circle*{1.61}}
\multiput(48.68,84.515)(.03369863,-.044383562){219}{\line(0,-1){.044383562}}
\multiput(48.5,63.995)(.033652174,.045391304){230}{\line(0,1){.045391304}}
\put(41.6,48.335){\makebox(0,0)[cc]{$10001(18)$}}
\put(61.74,48.695){\makebox(0,0)[cc]{$10101(22)$}}
\put(39.2,85.415){\makebox(0,0)[cc]{$01101(14)$}}
\put(38.56,63.095){\makebox(0,0)[cc]{$01110(15)$}}
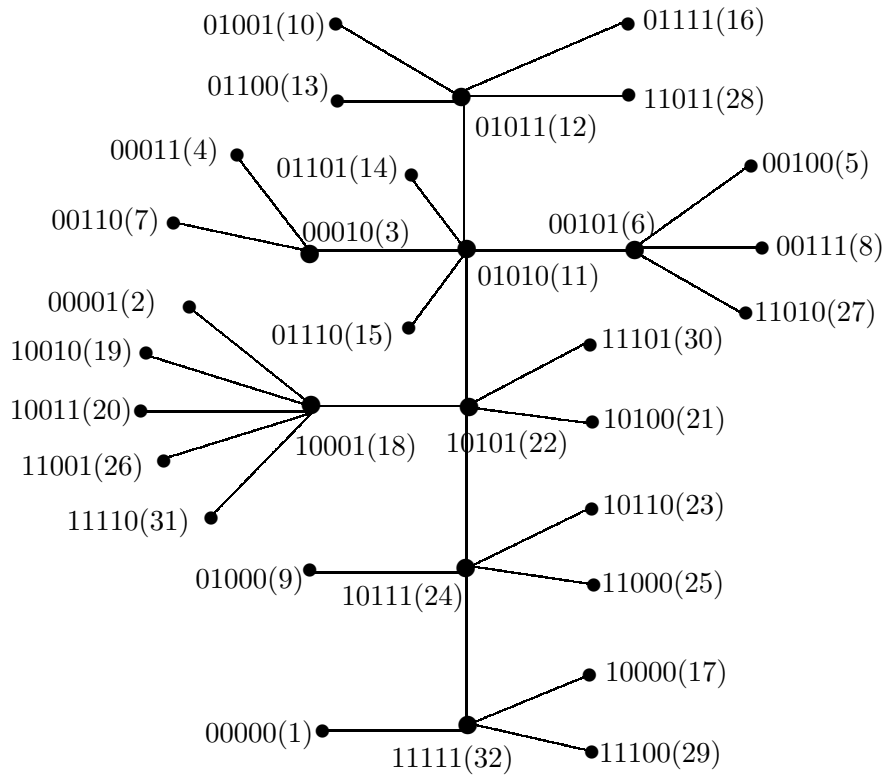
\put(55.88,1.175){\makebox(0,0)[cc]{Fig. 4(d). Spanning Tree $T_4$ in $AQ_5$}}
\end{picture}
\]

Here, we observe that\\
$InV(T_1)= \{1, 5, 7, 16, 17, 25, 27, 31\}$,\\
$InV(T_2) = \{2, 4, 8, 10, 21, 23, 26, 30\}$ \\
$InV(T_3)= \{9, 13, 14, 15, 19, 20, 28, 29\}$	\\
$InV(T_4)= \{3, 6, 11, 12, 18, 22, 24, 32\}$	\\
are such that $InV(T_i) \cap InV(T_j) = \phi $, for $i \neq j$ and $ 1 \leq i, j \leq 4$\\
Also, observe\\
$E(T_1) = \{\langle 1, 2   \rangle,  \langle 1, 3   \rangle,  \langle 1, 5    \rangle,  \langle 1, 16   \rangle,  \langle 1, 17   \rangle,  \langle  4, 5  \rangle,  \langle 5, 7   \rangle,  \langle 5, 12   \rangle,  \langle 5,21   \rangle,  \langle 5, 28   \rangle, \langle 6, 7   \rangle,  \langle 7, 8   \rangle, $

$\hspace{1cm}   \langle 7, 10   \rangle,  \langle 9, 25   \rangle,  \langle 11, 27   \rangle,  \langle 13, 16   \rangle,  \langle 14, 16   \rangle,  \langle 15, 16   \rangle,  \langle 17. 18   \rangle,  \langle 17, 20   \rangle,  \langle 17, 24   \rangle,  \langle 17, 25   \rangle,  $

$\hspace{1cm} \langle 19, 27   \rangle,  \langle 22, 27   \rangle,  \langle 23, 31   \rangle,  \langle 25, 27   \rangle,  \langle 26, 27   \rangle,  \langle 27, 31   \rangle,  \langle 29, 31   \rangle,  \langle 30, 31   \rangle,  \langle  31, 32  \rangle  \}     $\\
$E(T_2) = \{\langle 1, 4   \rangle,  \langle  2, 3  \rangle,  \langle  2, 4  \rangle,  \langle 2, 6   \rangle,  \langle 2, 10   \rangle,  \langle  4, 8  \rangle,  \langle 4, 12   \rangle,  \langle 4, 13   \rangle,  \langle 4, 29   \rangle,  \langle 5, 8   \rangle,  \langle 7, 23   \rangle, \langle 8, 16   \rangle,  $

$\hspace{1cm} \langle 8, 24   \rangle, \langle 8, 25   \rangle,  \langle 9, 10   \rangle,  \langle 10, 11   \rangle,  \langle 10, 15   \rangle,  \langle 10, 23   \rangle,  \langle 14, 30   \rangle,  \langle 17, 21   \rangle,  \langle 18, 23   \rangle,  \langle 19, 30   \rangle,  $

$\hspace{1cm} \langle 20, 21   \rangle, \langle 21, 23   \rangle,  \langle 21, 28   \rangle,  \langle 22, 23   \rangle,  \langle 23, 26   \rangle,  \langle 26, 30   \rangle,  \langle 26, 31   \rangle,  \langle  27, 30  \rangle,  \langle 30, 32   \rangle   \}     $\\
$E(T_3) = \{\langle 1, 9   \rangle,  \langle 2, 15   \rangle,  \langle  3, 14  \rangle,  \langle 4, 20   \rangle,  \langle  5, 13  \rangle,  \langle 6, 14  \rangle,  \langle 7, 15   \rangle,  \langle 8, 9   \rangle,  \langle 9, 11   \rangle,  \langle 9, 12   \rangle,  \langle 9, 13   \rangle, $

$ \hspace{1cm} \langle 9, 16   \rangle,  \langle 10, 14   \rangle,  \langle 13, 14   \rangle,  \langle 13, 15   \rangle,  \langle 13, 29   \rangle,  \langle 15, 18   \rangle,  \langle 15, 31   \rangle,  \langle 17, 19   \rangle,  \langle 19, 20   \rangle,  \langle 19, 22   \rangle,  $

$ \hspace{1cm} \langle 19, 23   \rangle,  \langle 20, 24   \rangle,   \langle 20, 29   \rangle,  \langle 21, 29   \rangle,  \langle 25, 29   \rangle,  \langle 26, 28   \rangle,  \langle 27, 28   \rangle,  \langle 28, 29   \rangle,  \langle 28, 32   \rangle,  \langle  29, 30  \rangle  \}     $\\
$E(T_4) = \{\langle 1, 32   \rangle,  \langle 2, 18   \rangle,  \langle 3, 4   \rangle,  \langle 3, 7    \rangle,  \langle 3, 11   \rangle,  \langle 5, 6   \rangle,  \langle 6, 8   \rangle,  \langle 6, 11   \rangle,  \langle 6, 27   \rangle,  \langle 9, 24   \rangle,  \langle  10, 12  \rangle, \langle 11, 12   \rangle,$

$ \hspace{1cm}   \langle 11, 14   \rangle,  \langle 11, 15    \rangle,  \langle 11, 22   \rangle,  \langle 12, 13   \rangle,  \langle 12, 16   \rangle,  \langle 12, 28   \rangle,  \langle 17, 32   \rangle,  \langle 18, 19   \rangle,  \langle 18, 20   \rangle,  $

$\hspace{1cm} \langle 18, 22   \rangle,  \langle 18, 26,  \rangle, \langle 18, 31   \rangle,  \langle 21, 22   \rangle,  \langle 22, 24   \rangle,  \langle 22, 30   \rangle,  \langle 23, 24   \rangle,  \langle 24, 25   \rangle,  \langle  24, 32  \rangle,  \langle  29, 32  \rangle   \}     $\\
are such that $ E(T_i) \cap E(T_j) = \phi$, for $i \neq j$ and $ 1 \leq i, j \leq 4$.\\ 

According to Lemma \ref{lm1}, above constructed trees $T_1, T_2, T_3$ and $T_4$ are CISTs on $AQ_5$.\\

\begin{theorem} Let $n\geq 6$ be an integer. There exist four completely independent spanning trees of which two are with diameter $2n-3$ and two are with diameter $2n-5$, in augmented cube $AQ_n$. 
\end{theorem} 

\begin{proof}  	By using above constructed four CISTs in $AQ_5$ and Corollary \ref{lm3}, we get four CISTs in $AQ_n$ for $n \geq 6$.\\
	As we want vertices to be close together to avoid communication delays. So, we will concentrate on diameters of above trees.\\
		Above constructed trees $T_1, T_2$ have diameters $8$ and $T_3, T_4$ have diameter $6$ in $AQ_5$. We first consider the construction of four CISTs in $AQ_6$. It is sufficient to show the construction of only one tree. Consider tree $T_1$ of $AQ_5$, having longest path of length $8$ and central vertex $10000(17)$. Now, by prefixing $0$ and $1$ to this vertex we get central vertex of $T^0_1$ and $T^1_1$ respectively. Means, we select $u_1 = 010000 \in V(T^0_1)$ and $v_1 = 110000\in V(T^1_1)$. Then, $\overline T_1$ is CIST with diameter $9$ on $AQ_6$. Constructing in the similar manner we get $\overline T_2$, $\overline T_3$, $\overline T_4$ CISTs with diameters $9$, $7$, $7$ respectively on $AQ_6$.    \\
	Let $AQ_{n+1}$ ($n \geq 6)$ be decomposed into two augmented cubes say $AQ^0_n$ and $AQ^1_n$ with
	vertex set say $ \{x^0_i : 1 \leq i \leq 2^{n}\}$ and  $\{x^1_i : 1 \leq i \leq 2^{n}\}$ respectively. 
	Denote by $T^0_1, T^0_2, T^0_3, T^0_4$  the CISTs with diameter $2n-3, 2n-3, 2n-5, 2n-5$ respectively in $AQ^0_n$. Let the identical corresponding CISTs
	in $AQ^1_n$ be denoted by $T^1_1, T^1_2, T^1_3, T^1_4$. \\
		Now, we will prove by induction that the diameters of $\overline T_1, \overline T_2, \overline T_3, \overline T_4)$ CISTs in $AQ_{n+1}$ ($n \geq 6$) are $2n-1, 2n-1, 2n-3, 2n-3$ respectively. \\
	 It is sufficient to prove result for a single tree say $T^0_1$.\\
	Let $P^0_1 = x^0_1-x^0_2-.......x^0_{n-1}-x^0_n-.....x^0_{2n-2} $ be the longest path naturally of length $2n-3$ in tree $T^0_1$. And $P^1_1 = x^1_1-x^1_2-.......x^1_{n-1}-x^1_n-.....x^1_{2n-2} $ be its corresponding path in corresponding tree $T^1_1$. \\
	The vertex $x^0_n$ is in the center of the path $P^0_1$ hence any vertex on tree $T^0_1$ will be within a distance $n-1$ from the vertex $x^0_n$, means any vertex $x^0_i \in V(T^0_1)$, $d_{T^0_1}(x^0_i, x^0_n ) \leq n-1$. Similarly, any vertex $x^1_i \in V(T^1_1)$, $d_{T^1_1}(x^1_i, x^1_n ) \leq n-1$. \\
	Now, consider tree say  $\overline T_1$ be a spanning tree of $AQ_{n+1}$ constructed from $T^0_1$ and $T^1_1$ by adding an edge $\langle x^0_n, x^1_n \rangle \in E(AQ_{n+1})$ to connect two internal vertices $x^0_n \in V(T^0_1)$ and $x^1_n \in V(T^1_1)$. Then, $\overline T_1$ is with diameter $2n-1$. As $V(\overline T_1) = V(T^0_1) \cup V(T^1_1)$, $d_{\overline T_1}(x^0_i, x^1_j )= d_{T^0_1}(x^0_i, x^0_n ) + 1 + d_{T^1_1}(x^1_n, x^1_j ) \leq (n-1) + 1 + (n-1) = 2n-1$.\\
	Constructing in the similar manner we get $\overline T_2$, $\overline T_3$, $\overline T_4$ CISTs with diameters $2n-1$, $2n-3$, $2n-3$ respectively on $AQ_{n+1}$.    \\

\end{proof}

{\bf Concluding remarks.}\\
In this paper, we have proposed a construction of four CISTs in augmented cube $AQ_n$ ($n \geq 6$) of which two trees with diameters $2n-3$ and two trees with diameters $2n-5.$
Our results provide $n-1$ CISTs in augmented cube $AQ_n$ ($n= 3,4,5$) and thus we pointed out that Hasunuma's conjecture does hold for $AQ_n$ when $n= 3,4,5$. As connectivity of augmented cubes is comparatively higher than other variants of hypercubes, an interesting problem is whether Hasunuma's conjecture is true for $AQ_n$ ($n\geq 6$) if so then how to derive an algorithm that construct $n-1$ CISTs in $AQ_n$ ($n\geq 6$)? \\  

\noindent {\bf Acknowledgment:} The first author gratefully
acknowledges the Department of Science and Technology, New Delhi, India
for the award of Women Scientist Scheme for research in Basic/Applied Sciences.

$$\diamondsuit\diamondsuit\diamondsuit$$

\end{document}